\documentclass[12pt]{amsart}
\usepackage{amssymb,amsmath,amsmath}
\input{amssym.def}
\input{amssym.tex}
\def\newtheorems{\newtheorem{theorem}{Theorem}[section]
                 \newtheorem{corollary}[theorem]{Corollary}
                 
                 \newtheorem{lemma}[theorem]{Lemma}
                 \newtheorem{claim}[theorem]{Claim}

                 \newtheorem{fact}[theorem]{Fact}
                 \newtheorem{definition}[theorem]{Definition}
                 \newtheorem{example}[theorem]{Example}

                 \newtheorem{definition and lemma}[theorem]{Definition and
                 Lemma}}

\newtheorems

\newcommand{\gb}{{\frak b}}

\newcommand{\B}{\mathcal{B}}

\newcommand{\F}{\mathcal{F}}

\newcommand{\ON}{\mathcal{ON}}

\newcommand{\Q}{\mathcal{Q}}
\newcommand{\R}{\mathcal{R}}
\renewcommand{\S}{\mathcal{S}}
\newcommand{\T}{\mathcal{T}}
\newcommand{\U}{\mathcal{U}}
\newcommand{\V}{\mathcal{V}}
\newcommand{\W}{\mathcal{W}}

\newcommand{\cf}{\mbox{\rm cf}}
\newcommand{\dom}{\mbox{\rm dom}}
\newcommand{\ran}{\mbox{\rm ran}}
\newcommand{\type}{\mbox{\rm type}}
\newcommand{\ord}{\mbox{\rm ord}}

\begin{document}
\title{Orders of $\pi$-bases}
\author { Isaac Gorelic}
\address{Government of Canada, Ottawa, Ontario, Canada}

\email{isaacgorelic@yahoo.com }
\thanks{Research supported by NSERC of Canada.}

\date{\today}

\begin{abstract}
We extend the scope of B. Shapirovskii's results [6] on the order
of $\pi$-bases in compact spaces and answer some questions of V.
Tkachuk in [7].
\end{abstract}

\maketitle
\section*{Introduction}
The notion of $\pi$-base is an essential tool for studying the
internal structure of a topological space as well as its external
properties (embeddings, functions and the like); this was
established, primarily, in the work of Boris Shapirovskii in the
1970s ([5] and [6] containing major discoveries). In this paper we
attempt to show the full natural scope of his ideas regarding the
order of $\pi$-bases.

In Section 1, we decipher and refine the method of induction used
by Shapirovskii in Section 3 of [6]. We develop a purely
set-theoretic technique which will be applied to generalize
Shapirovskii's results and answer some questions of Tkachuk; this
technique provides a formalism interesting in itself and apt to
have other applications.

In Section 2, using the results of Section 1, we describe a
canonical form for $\pi$-bases in regular spaces and prove that
canonical $\pi$-bases always exist. In our Lemma 2.4 we give a
characterization of free sequences and with its help we derive a
series of new results, starting with the central Theorem 2.6. By
carefully numbering the points of our argument, we have tried to
achieve ``sufficiency with precision," and, hopefully, to avoid
accidental and irrelevant conditions in our statements. This also
gave us reliable guidance as to where exactly to look for
counterexamples when the sufficiency of a weaker condition was in
question.

Section 3 deals with the natural question as to whether or not the
assumptions in our theorems could be further relaxed. We give some
examples to the contrary which also solves three problems of V.
Tkachuk from [7].

The idea for this paper originated from the observation that our
Lemma 2.4 could be used, in place of final compactness (that is,
small $L(X)$), even in the original Shapirovskii argument, made
for compact spaces.

The author is grateful to Vladimir Tkachuk for his stimulating
influence and to Stevo Todor\v{c}evi\'c for advice. Special thanks
go to the referee who suggested numerous improvements to the
paper.

 We have used [1] and [2] as general references for definitions and
notation.  $\ON$ is the class of ordinal numbers. Additions and
multiplications are ordinal operations. We have written
$[\gamma,\delta )$, or $\delta \setminus \gamma$, for $\{\alpha :
\gamma \leq \alpha < \delta\}$. We have denoted by ${\mathcal
{T}}_X$ the family of all non-empty open subsets of a topological
space $X$.
\medskip

\section{Canonical $\kappa$-functions}
\begin{definition}
For an infinite cardinal $\kappa$, a {\bf canonical
$\kappa$-function} is a class function    $$\phi = \phi_\kappa :
\ON \longrightarrow \ [\ON \times \kappa ]^{<\omega}$$ satisfying
the following two conditions:
\begin{enumerate}
\item
For every ordinal $\alpha$, $\phi (\alpha ) \subseteq \alpha
\times\kappa $.
\item
For every ordinal $\delta$ of the form $\delta  =
\kappa\cdot\epsilon$ there is $\gamma(\delta) < \delta$ such that
$[\ [\gamma(\delta),\delta )\times \kappa ]^{<\omega} \subseteq
\phi``\delta$.
\end{enumerate}
\end{definition}

\begin{definition}
A {\bf $\tau$-strong canonical $(\kappa , \lambda) $-function} is
a function $$\psi : \lambda\longrightarrow \ [\lambda \times
\kappa ]^{\tau}$$ satisfying the following two conditions:
\begin{enumerate}
\item
$(\forall \alpha \in \dom (\psi))$  $\psi (\alpha ) \subseteq
\alpha \times\kappa $.
\item
For every ordinal $\delta \leq\lambda$ with $\cf(\delta ) =
\kappa^+$ there is $\gamma(\delta) < \delta$ such that $[\
[\gamma(\delta),\delta )\times \kappa ]^{\tau} \subseteq
\psi``\delta$.
\end{enumerate}
\end{definition}

\begin{definition and lemma}
Let $\kappa$ be an infinite cardinal. Define a class-function
$\sigma = \sigma_\kappa : \ON \longrightarrow \ON$ by the
following rule:
\begin{itemize}
\item
$\sigma(0) = 0$,
\item
$\sigma(1) = \kappa$,
\item
$\sigma(\alpha + 1) = \sigma (\alpha) + |\sigma (\alpha)|$, for
$\alpha > 0$,
\item
$\sigma(\beta) = {\rm sup} \{\sigma (\alpha):\alpha < \beta \}$,
for $\beta$ limit.
\end{itemize}
Then every ordinal $\delta$ has the following unique
$\sigma_\kappa$-normal form: $$ \delta = \sigma (\alpha_0 ) +
\sigma ( \alpha_1 )+ \dots + \sigma (\alpha_{n-1}) + \Delta,$$
where $n \in \omega$, $|\sigma (\alpha_0)|
> |\sigma(\alpha_1)| > |\sigma (\alpha_2)| > \dots  > |\sigma(\alpha_{n-2})|
> |\sigma (\alpha_{n-1})|, \alpha_{n-1} > 0$, and $\Delta < \kappa$.
\end{definition and lemma}
\begin{proof}
To visualize, we partition $\ON$ into intervals $[0],
[1,\kappa^+),\dots, [\mu ,\mu^+), \dots$. This is the finest
partition of $\ON$ into intervals that are closed under $\sigma$.
Then we choose descending $\alpha_i$ from different intervals,
excluding the first.

Existence. Since $\sigma$ is increasing continuous, and $\sigma(1)
= \kappa$, if $\delta \geq \kappa$, then $\exists ! \alpha_0 > 0 $
such that $\sigma(\alpha_0) \leq \delta < \sigma (\alpha_0 + 1)$.
Similarly, if $\type(\delta \setminus \sigma(\alpha_0)) \geq
\kappa$, then $\exists ! \alpha_1> 0$ such that $\sigma(\alpha_1)
\leq \type(\delta \setminus \sigma(\alpha_0)) < \sigma(\alpha_1
+1)$. Eventually, we'll get to $\alpha_{n-1} > 0$ (if any,
otherwise set $n=0$) such that $\sigma( \alpha_{n-1}) \leq \type
(\delta \setminus (\sigma (\alpha_0 ) + \sigma ( \alpha_1 )+ \dots
+ \sigma (\alpha_{n-2})) < \sigma(\alpha_{n-1} +1)$, but now
$\type (\delta \setminus (\sigma (\alpha_0 ) + \sigma ( \alpha_1
)+ \dots + \sigma (\alpha_{n-1})) < \kappa$. Put $\Delta = \type
(\delta \setminus (\sigma (\alpha_0 ) + \sigma ( \alpha_1 )+ \dots
+ \sigma (\alpha_{n-1}))$.

Uniqueness is now easily proved by induction on the length of the
normal form. It follows that the lexicographic ordering of the
$\sigma$-normal forms (that is of the ordinal sequences
$<\alpha_0,\alpha_1,\dots,\alpha_{n-1},\Delta>$) coincides with
the natural order of their values in $\ON$, but we will not need
this explicitly.
\end{proof}

\begin{definition}
Define a total pressing-down (save for $\gamma (0) =0$)
class-function $\gamma = \gamma_\kappa: \ON \longrightarrow \ON$
as follows:

For every ordinal $\delta$ with the $\sigma_\kappa$-normal form
$\delta = \sigma (\alpha_0 ) + \sigma ( \alpha_1 )+ \dots +
\sigma(\alpha_{n-2}) + \sigma (\alpha_{n-1}) + \Delta$, set

\begin{itemize}
\item
$\gamma (\delta ) = 0$, if $n=0$,
\item
$\gamma (\delta ) =  \sigma (\alpha_0 ) + \sigma ( \alpha_1 )+
\dots + \sigma(\alpha_{n-2})$, otherwise.
\end{itemize}
\end{definition}

\begin{theorem}
For every infinite cardinal $\kappa$ there is a canonical
$\kappa$-function $\phi = \phi_\kappa$.
\end{theorem}

\begin{proof}
1) For every ordinal $\delta$ with $\Delta = 0$ in its normal
form, let $\delta'= \gamma (\delta) + \sigma (\alpha_{n-1} + 1)$
(for $\delta > 0$ this is also $\delta' = \delta + |\sigma
(\alpha_{n-1})|$). Then fix a function $f_\delta : [\delta ,
\delta') \longrightarrow [[\gamma(\delta) , \delta') \times \kappa
]^{<\omega}$ such that
\begin{enumerate}
\item
$f_\delta$ is onto, and
\item
$\forall \xi \in [\delta, \delta')$  $f_\delta ( \xi ) \subseteq
\xi \times \kappa$.
\end{enumerate}
This is very easy to arrange, because, for every $\alpha$,
$|\sigma (\alpha+1)| = | [\sigma (\alpha), \sigma(\alpha + 1))|
\geq \kappa$. We may start with an arbitrary surjection mapping
$|\sigma(\alpha_{n-1} + 1)) |$-many times to every member of the
range.

2) Now consider ordinals $\delta$ of the form
$\delta=\kappa\cdot\epsilon$. These are the same as just
considered ordinals with $\Delta = 0$ in their normal form.

Suppose that we have finitely many functions $h_0, \dots, h_{n-1}$
such that $(\forall i < n)$ $h_i : [\delta , \delta + \kappa)
\longrightarrow [(\delta +\kappa) \times \kappa ]^{<\omega}$ and
$(\forall \xi \in \dom (h_i))$ $h_i (\xi) \subseteq \xi \times
\kappa$. Then denote by $H=H[h_0,\dots, h_{n-1} ]$, and fix, a
function with the same domain and co-domain such that $(\forall i)
(\forall \xi \in [\delta , \delta + \kappa))$ $(\exists \eta \geq
\xi)$ $H(\eta) = h_i (\xi)$ (and so $H(\eta) = h_i(\xi) \subseteq
\xi \times \kappa \subseteq \eta \times \kappa$). In other words,
$H$ is a combination of $h_0, \dots , h_{n-1}$ mapping onto the
union of their ranges.

3)Finally, define $\phi = \phi_\kappa$ on the ordinal intervals of
the form $[\delta, \delta + \kappa)$ with $\delta = \kappa \cdot
\epsilon$, simultaneously for all such $\delta$, by the following
explicit rule. Find the normal form $\kappa\cdot\epsilon = \delta
= \sigma (\alpha_0 ) + \sigma ( \alpha_1 )+ \dots + \sigma
(\alpha_{n-1}) + 0$. Then set $\phi \restriction [\delta , \delta
+ \kappa) = H[h_0, \dots , h_{n-1}]$, where $h_i = f_{\sigma
(\alpha_0 ) + \sigma ( \alpha_1 )+ \dots + \sigma (\alpha_{i})}
\restriction [\delta, \delta +\kappa)$.

4)We are left to check that the function $\phi$ just defined
satisfies the Definition 1.1. It is transparent that the first
condition is satisfied, and the second is in the following
assertion.
\begin{claim}
Suppose that, for every $\delta$ with $\Delta = 0$ (and $n\geq 0$)
in its normal form, $$ \ran (f_\delta ) \subseteq \ran (\phi
\restriction [\delta , \delta' )).$$ Then, for every such $\delta$
with $n \geq 1$, $$[[\gamma (\delta ), \delta )\times
\kappa]^{<\omega} \subseteq \ran (\phi \restriction [\gamma
(\delta ) , \delta )).$$
\end{claim}

\begin{proof}
This is straightforward by induction on $n$, and then by a
subinduction on $\alpha_{n-1}$ in the normal form for $\delta$.

The case $\alpha_{n-1} = \beta + 1$ is explicit, and for
$\alpha_{n-1}$ a limit ordinal use $[[\gamma (\delta), \gamma
(\delta) + \sigma (\alpha_{n-1}))\times \kappa ]^{<\omega} =
\bigcup_{\beta < \alpha_{n-1}} [[\gamma (\delta), \gamma (\delta)
+ \sigma (\beta))\times \kappa ]^{<\omega}$. The equation is true,
because $\{\beta < \alpha_{n-1} : \gamma (\gamma(\delta) + \sigma
(\beta) ) = \gamma (\delta)\}$ is cofinal in $\alpha_{n-1}$.
\end{proof}
\end{proof}

\begin{theorem}
If $(\kappa^+)^\kappa  = \kappa^+$ and, for every cardinal $\mu$
with $\kappa^+ \leq\mu < \lambda$, we have $\mu^\kappa=\mu$, then
there is a $\kappa$-strong $(\kappa , \lambda) $-function. Under
CH, there is an $\omega$-strong $(\omega ,
\aleph_\omega)$-function.
\end{theorem}

\begin{proof}

This time consider $\sigma_{\kappa^+}$-normal forms for ordinals
$\delta \in \ON$ and the regressive function $\gamma_{\kappa^+}$.
Otherwise do - {\it mutatis mutandis} - as in the proof of Theorem
1.5. It will be still possible to define $f_\delta : [\delta ,
\delta') \twoheadrightarrow [[\gamma(\delta) , \delta') \times
\kappa ]^{\kappa}$, because, for every $\delta$ of the form
$\delta = \kappa^+ \cdot \epsilon$, we have \linebreak $|[\delta,
\delta')|^\kappa = |[\delta, \delta')|\geq \kappa^+$.
\end{proof}

\section{Shapirovskii $\pi$-bases in regular spaces}

Recall ([1]) that $\R \subseteq \T_X$ is a $\pi$-base for the
topology $\T_X$ of $X$ iff $\forall U \in \T_X$ $\exists R \in \R$
with $R\subseteq U$. A family $\R \subseteq \T_X$ is a local
$\pi$-base for $p \in X$ iff $\forall U \in \T_X$ with $p\in U$
$\exists R \in \R$ with $R\subseteq U$.  The $\pi$-character of a
point $p$ in $X$ is the cardinal $\pi\chi (p,X)= \min \{ |\R| : \R
\subseteq \T_X$ is a local $\pi$-base for $p \}$ and the
$\pi$-character of $X$ is $ \pi\chi (X) = \sup \{ \pi\chi (p,X) :
p \in X \}$.

For $P= \{p_\alpha : \alpha < \mu \}$ and $\delta < \mu$, let's
write $P_\delta = \{ p_\alpha : \alpha < \delta \}$ and $P^\delta
= \{ p_\alpha : \delta \leq \alpha < \mu \}$.

$P= \{p_\alpha : \alpha < \mu \} \subseteq X$ is called {\bf
left-separated}, if $\overline{P_\delta} \cap P^\delta =
\emptyset$, for every $\delta < \mu$ (this means that all initial
segments of $P$ are relatively closed in $P$). It is well-known
and easy to see that every space $X$ has a dense subspace
left-separated in the order-type $d(X)$.

\begin{definition}
Suppose $X$ is a regular topological space with $\pi\chi (X)=
\kappa$. Suppose the density of $X$, $d(X) = \lambda$, is an
infinite cardinal and $P= \{p_\alpha : \alpha < \lambda \}$, is a
left-separated dense subspace, left-separated as written. Then $\S
= \{S_{\alpha,i} : <\alpha,i> \in \lambda \times \kappa \}
\subseteq \T_X$, or $\S$ together with $P$, is called a {\bf
Shapirovskii $\pi$-base for $X$} if the following conditions are
satisfied:

\begin{itemize}
\item[(a)]
$\{S_{\alpha , i} : i < \kappa\}$ a local $\pi$-base for
$p_\alpha$ in $X$;
\item[(b)]
$\overline{P_\alpha} \cap \cup \{ \overline{S_{\beta, i}}: \alpha
\leq \beta , i < \kappa \} = \emptyset$.
\item[(c)]
$(\forall \delta = \kappa\cdot\epsilon) (\forall A \in [[\gamma
(\delta), \delta) \times \kappa]^{<\omega})$ if $\bigcap_{a \in A}
\overline{S_a} \neq \emptyset$, then $\bigcap_{a \in A}
\overline{S_a} \cap \bigcup \{ \overline{P_\alpha} : \alpha <
\delta\} \neq \emptyset$, and therefore $\bigcap_{a \in A}
\overline{S_a} \cap \overline{P_\delta} \neq \emptyset$.
\end{itemize}
\end{definition}

We will also say that $\S$ as above is a {\bf $\kappa$-strong
Shapirovskii $\pi$-base}, if the condition (c) is replaced by the
following:

(c*) $(\forall \delta = \kappa^+ \cdot\epsilon) (\forall A \in
[[\gamma (\delta), \delta) \times \kappa]^{\kappa})$ if
$\bigcap_{a \in A} \overline{S_a} \neq \emptyset$, then
$\bigcap_{a \in A} \overline{S_a} \cap \overline{P_\delta} \neq
\emptyset$.

\begin{theorem}
Every regular space has a Shapirovskii $\pi$-base.
\end{theorem}

\begin{proof}
Let $Q= \{ q_\alpha : \alpha < \lambda\}$ be a left-separated
dense subspace of $X$. Define $P=\{p_\alpha : \alpha < \lambda \}$
and $\S = \{S_{\alpha,i} : <\alpha,i> \in \lambda \times \kappa
\}$ by induction on $\delta < \lambda$.

Suppose at the stage $\delta \geq 0$ we have  $P_\delta =
\{p_\alpha : \alpha < \delta \}$ and $\S = \{S_{\alpha,i} :
<\alpha,i> \in \delta \times \kappa \}$.

1) If $\bigcap \{ \overline{S_a} : a \in \phi_\kappa (\delta) \}
\cap \overline{P_\delta} = \emptyset$ and $\bigcap \{\overline{
S_a} : a \in \phi_\kappa (\delta) \} \neq \emptyset$, pick
$p_\delta$ in $\bigcap \{\overline{ S_a} : a \in \phi_\kappa
(\delta) \}$.

Otherwise, put $p_\delta = q_\xi$, where $\xi$ is the least index
of a member of $Q \setminus \overline{P_\delta}$.

2) Next, $p_\delta$ being thus defined, pick a $\pi$-base $\B$ for
$p_\delta$ of size $|\B|\leq \kappa$ and with $\overline{P_\delta}
\cap \overline{B} =\emptyset$ for each member $B$ of $\B$, - and
index it as $\{S_{\delta , i} : i \in \kappa \}$. End of
induction.
\end{proof}

\begin{theorem}
Under the cardinal assumptions of Theorem 1.7, every regular space
with $\pi\chi (X) \leq \kappa$ and $d(X) \leq \lambda$ has a
$\kappa$-strong Shapirovskii $\pi$-base.
\end{theorem}

Recall ([1]) that $P= \{p_\alpha : \alpha < \mu \} \subseteq X$ is
a {\bf free sequence} in the space $X$, if $\overline{P_\delta}
\cap \overline{P^\delta} = \emptyset$, for every $\delta < \mu$.
Let $\F (X) =$ sup $\{ | P| : P$ is a free sequence in $X \}$.

The following characterization of free sequences parallels
Shapirovskii's characterization of discrete sets in [4].\footnote
{For an important different (external, or ``algebraic") approach
to free sequences see [9] and [10].} It says that small $\F(X)$
can be viewed as a compactness-like reflection property of the
space $X$, - and this is precisely what we will need in the
sequel.

\begin{lemma}
Let $X$ be any topological space and $\kappa$ any infinite
cardinal. Then (a) $\F(X) \leq \kappa$ if and only if (b) for
every $Y\subseteq X$, every family $\U \subseteq T_ X$ such that
$(\forall A \in [Y]^{\leq\kappa})$ $(\exists U \in \U)$
$\overline{A} \subseteq U$ has a subfamily $\V \subseteq \U$ of
size $|\V| \leq \kappa$ covering $Y$.
\end{lemma}
\begin{proof}
{\it Sufficiency.} Suppose $Y$ and $\U$ are as in (b), but (b)
fails, so there is no $\V \in [\U]^{\leq\kappa}$ covering $Y$. We
will pick up a free sequence $P= \{p_\alpha : \alpha < \kappa^+
\}$ by induction on $\delta < \kappa^+$. Suppose that at the stage
$\delta \geq 0$ we have $P_\delta = \{ p_\alpha : \alpha <
\delta\}$ and $\{U_\alpha : \alpha < \delta \} \subseteq \U$.

Then pick $U_\delta \in \U$ with $\overline{P_\delta} \subseteq
U_\delta$ and $p_\delta \in Y \setminus \bigcup_{\alpha \leq
\delta} U_\alpha$.

We claim that $P$ is a free sequence. Indeed, for $\delta <
\kappa^+$, $\overline{P_\delta} \subseteq U_\delta$ and $P^\delta
\bigcap U_\delta = \emptyset$, whereupon $\overline{P_\delta}
\bigcap \overline{P^\delta} = \emptyset$.

{\it  Necessity.} Now assume that (a) fails, and there is a free
sequence $P= \{p_\alpha : \alpha < \kappa^+ \}$. Let $\U = \{
X\setminus \overline{P^\delta} : \delta < \kappa^+ \}$. Because
$\kappa^+$ is a regular cardinal, $(\forall A \in [P]^{\leq
\kappa})$ $(\exists \delta < \kappa^+)$ $\overline{A} \subseteq
\overline{P_\delta} \subseteq X\setminus \overline{P^\delta}$.
Therefore, $\U$ is as in (b) with $Y=P$. Let $\V$ be a subfamily
of $\U$ of size $|\V| \leq \kappa$. Then, again by regularity of
$\kappa^+$, $(\exists \delta < \kappa^+)$ such that $\bigcup \V
\subseteq \bigcup_{\gamma \leq \delta} (X\setminus
\overline{P^\gamma}) = X\setminus \overline{P^\delta}$. Therefore
$\bigcup \V$ is disjoint from $P^\delta$, and thus does not cover
$P$.
\end{proof}

Recall (see [0]) that a space $X$ is {\bf initially
$\kappa$-compact}, if every cover of cardinality at most $\kappa$
has a finite subcover, and $X$ is {\bf
$(\kappa,\kappa^+]$-compact}, if every cover of $X$ of cardinality
$\kappa^+$ has a subcover of cardinality $\kappa$.

The fact that $t(X) = \F(X)$ in compact spaces is well known, but
we will need these weaker covering properties of $X$ as a factor
in the relationship between $t(X)$ and $\F(X)$. The following fact
is folklore.

\begin{fact}
Suppose that $X$ is a regular space. Then
\begin{itemize}
\item[(a) ]
$X$ is initially $\kappa$-compact +  $\F(X) \leq \kappa$
$\Longrightarrow$ $t(X) \leq \kappa$.
\item[(b )]
$X$ is $(\kappa, \kappa^+]$-compact + $t(X) \leq \kappa$
$\Longrightarrow$ $\F(X) \leq \kappa$.
\end{itemize}
\end{fact}
\begin{proof}
(a) 1) Let $Y \subseteq X$ with $Y = \bigcup \{\overline{A} : A
\in [Y]^{\leq\kappa} \}$ and observe that $Y$ is initially
$\kappa$-compact.

2) It is sufficient to show that $Y$ is closed. So fix $p\notin
Y$. We will find a neighbourhood of $p$ disjoint from $Y$.

3) For every $A\in [Y]^{\leq\kappa}$ find $U_A$, a neighbourhood
of $p$ with $\overline{U_A} \cap \overline{A} = \emptyset$.

4) $\U = \{X\setminus \overline{U_A} : A \in [Y]^{\leq\kappa} \}$
is a cover of $Y$ as in (b) of Lemma 2.4.

5) Therefore, there is a $\V \subseteq \U$, $|\V| \leq \kappa$,
$\V$ covers $Y$, and since $Y$ is initially $\kappa$-compact, a
finite $\W \subseteq \V \subseteq \U$ which also covers $Y$.

6) But then $\bigcap \{U_A : X\setminus \overline{U_A} \in \W \}$
is a neighbourhood of $p$ disjoint from $Y$, as wanted.

(b) Suppose $P= \{p_\alpha : \alpha < \kappa^+ \} \subseteq X$. By
tightness, $\overline{P} = \bigcup \{ \overline{P_\alpha} : \alpha
< \kappa^+ \}$. If $P$ were a free sequence in $X$, then $\U = \{
X\setminus \overline{P^\delta} : \delta < \kappa^+ \}$ would be an
increasing $\kappa^+$-cover of $X$, contradicting $(\kappa,
\kappa^+]$-compactness of $X$.
\end{proof}

For a family $\R$ of subsets of $X$ and a point $p\in X$, the {\bf
order of $p$ in $\R$} is the cardinal $\ord(p,\R) = | \{ R \in \R
: p\in R \} |$. The {\bf order of $\R$} is $\ord (\R ) = \sup \{
\ord (p, \R) : (p\in X \}$. Finally, a family $\R$ is is {\bf
point-$\kappa$} if $\ord (\R) \leq \kappa$, i.e. if every point
belongs to at most $\kappa$ members of $\R$.

\begin{theorem}
Suppose $X$ is a regular initially $\kappa$-compact space with
$\pi\chi (X) = \kappa$ and no free sequences of length $\kappa^+$
(that is $\F(X) \leq \kappa$). Then any Shapirovskii $\pi$-base is
point-$\kappa$.
\end{theorem}

\begin{proof}
Let $\S$ be a Shapirovskii $\pi$-base for $X$, as displayed in the
definition, and suppose that $\R \subseteq \S$, $|\R| = \kappa^+$.
We have to show that $\bigcap \R = \emptyset$.

1) For some $I \in [\lambda \times \kappa]^{\kappa^+}$, $\R =
\{S_{\alpha , i} : <\alpha , i> \in I \}$, and so $|\pi_0 `` I|
=\kappa^+$, where $\pi_0$ denotes the projection from the square
to the first coordinate, $\pi_0 (<a,b>)= a$.

2) Pick $\delta\in\lambda$, the least ordinal such that  $|(\pi_0
``I) \bigcap \delta| =\kappa^+$. Then $\cf (\delta) =\kappa^+$.
Let $J = I \bigcap ([\gamma (\delta), \delta ) \times \kappa)$.
Then $|\pi_0``J| = \kappa^+$, because $\gamma(\delta) < \delta$.

Let $\Q = \{S_{\alpha , i} : <\alpha, i > \in J \} \subseteq \R$.

3) By Fact 2.5 (a), $t(X) \leq \kappa$.

4) Since $t(X) \leq \kappa < \kappa^+ = \cf (\delta)$,
$\overline{P_\delta} = \bigcup \{ \overline{P_\alpha} : \alpha <
\delta\}$.

5) Since, by the choice of $\delta$, $\pi_0``J$ is cofinal in
$\delta$, $\bigcap \{ \overline{Q} : Q\in \Q \} \bigcap
\overline{P_\delta} = \emptyset$. This uses 4) above and the
property (b) of Definition 2.1.

6) Therefore, $\U = \{ X \setminus \overline{Q} : Q \in \Q \}$ is
an open cover of $\overline{P_\delta}$, and $(\forall A \in
[\overline{P_\delta} ]^{\leq\kappa})$ $(\exists Q \in \Q )$ with
$\overline{ A} \subseteq X\setminus \overline{Q}$.

7) Since $\F(X) \leq \kappa$, Lemma 2.4 applies (with $Y =
\overline{P_\delta}$), and $\exists \V \subseteq \U$, $|\V| \leq
\kappa$, such that $\V$ is also a cover of $\overline{P_\delta}$.

8) Since $ \overline{P_\delta}$ is initially $\kappa$-compact,
there if a finite $\W \subseteq \V \subseteq \U$ which is a cover
of $\overline{P_\delta}$. Say, $\W = \{ X\setminus \overline{ S_a}
: a \in A\}$, for some finite $A \subseteq J$, and so we have
$(\bigcap_{a \in A} \overline{S_a}) \bigcap \overline{P_\delta} =
\emptyset$.

9) Since $\delta$ with $\cf (\delta) =\kappa^+$ is {\it a
fortiori} of the form $\delta=\kappa\cdot\epsilon$ for some
$\epsilon$, this implies (by the property (c) in the Definition
2.1) that $\bigcap_{a \in A} \overline{S_a} = \emptyset$, and
therefore $\bigcap_{a \in A} S_a = \emptyset$. But $\{ S_a : a \in
A \} \subseteq \Q \subseteq \R$, and so $\bigcap \R = \emptyset$.
\end{proof}

The numbering of points in this proof, and in some of our other
proofs, hopefully, helps to see exactly what ideas are involved
(and seem needed, but maybe are not), at each point of the
argument.

\begin{corollary}
Every regular countably compact space with countable
$\pi$-character and no uncountable free sequences has a
point-countable $\pi$-base.
\end{corollary}

This corollary gives a partial positive answer to Problem 4.4 of
[7].

The core of our argument, in fact, also proves the following
variations.

\begin{corollary}
Every regular initially $t(X)^+$-compact space with $\kappa =$
max$\{\pi\chi (X), t(X) \}$ has a point-$\kappa$ $\pi$-base.
\end{corollary}

\begin{corollary}
Every first-countable initially $\omega_1$-compact regular space
has a point-countable $\pi$-base.
\end{corollary}

\begin{corollary}
Let $\kappa =$ max$\{\pi\chi (X), t(X) \}$. If $d(X) \leq
\kappa^+$, then $X$ has a point-$\kappa$ $\pi$-base.
\end{corollary}
This is, in essence, Tkachuk's theorem 3.2 of [7].

\begin{corollary}
Suppose $X$ is a regular space which is initially $\F
(X)$-compact. Let $\kappa =$ max$\{\F (X) , \pi\chi (X) \}$. Then
$X$ has a point-$\kappa$ $\pi$-base. In fact, any Shapirovskii
$\pi$-base is point-$\kappa$.
\end{corollary}

\begin{corollary}
Every regular countably compact space with no uncountable free
sequences has a point-$\pi\chi (X)$  $\pi$-base.
\end{corollary}

In the presence of a nice cardinal arithmetic, the covering
restrictions can be altogether omitted, but only when the density
of the space is not too large.

\begin{theorem}
Suppose that $\kappa$ and $\lambda$ are cardinals such that
$(\kappa^+)^\kappa  = \kappa^+$ and, for every $\mu$ with
$\kappa^+ \leq\mu < \lambda$, $\mu^\kappa=\mu$. Then every regular
space $X$ with $\pi\chi (X) \leq \kappa$, $\F(X)\leq \kappa$  and
$d(X)\leq \lambda$ has a point-$\kappa$ $\pi$-base.
\end{theorem}
\begin{corollary}
Under $CH$, every regular space with $ \F(X) = \omega$ and
$d(X)\leq \aleph_\omega$ has a point-$\pi\chi (X)$ $\pi$-base.
\end{corollary}
\begin{corollary}
Under $CH$, every regular first-countable space with $d(X)\leq
\aleph_\omega$ and no uncountable free sequences has a
point-countable $\pi$-base.
\end{corollary}

\section{Counterexamples to weaker assumptions}
The main tool for proving that a space does not have a
point-$\kappa$ $\pi$-base is Shapirovskii's Theorem 3.2 in
[6]\footnote{If we define $m (X) = \min \{\sup \{ (\ord (p, \R))^+
: p\in X \} : \R$ is a $\pi$-base for $X \ \}$, it states that
$d(X) \leq m(X) \cdot s(X)$.} which we will need in the following
weak form:
\begin{itemize}
\item[($\star$)]
``If max$\{\kappa^+ , s(X)\} < d(X)$, then $X$ does not have a
point-$\kappa$ $\pi$-base."
\end{itemize}

This criterion is used in every example below.
\begin{example}
There is, in $ZFC$, a first-countable zero-dimensional
left-separated space $X$ such that $d(X) = |X| \geq (\beth_\omega
)^+$, $hL(X) = \beth_{\omega}$, and hence $ s(X) =\beth_{\omega}$.
\end{example}

By $(\star)$, $X$ cannot have a point-countable $\pi$-base. This
is one of the celebrated generalized $L$-spaces of Stevo
Todor\v{c}evi\'c (cf. Theorem 16 in [8]). This example gives a
negative answer to Problem 4.1 of [7].

\begin{example}
There is, in $ZFC$, a zero-dimensional first-countable space
left-separated in the order-type $\gb$ with no discrete subspace
of size $\gb$. It has a point-countable $\pi$-base if and only if
$\gb = \omega_1$.
\end{example}
This is another $L$-space of Stevo Todor\v{c}evi\'c from the same
paper. In the case of $\gb = \omega_1$, whatever the value of
$\F(X)$ is, the space has a point-countable $\pi$-base by 2.10.

\begin{example}
Consistently, relative to the existence of a supercompact
cardinal, there is a first-countable hereditarily Lindel\"{o}f
(hence with $\F(Y) = s(Y) = \omega$) space $Y$ left-separated in
the order-type $\omega_2 = 2^\omega$, without a point-countable
$\pi$-base.
\end{example}

There is, in $ZFC$, a zero-dimensional space $X$ left-separated in
the order-type $d(X) = |X| = 2^{\beth_{\omega_1}}$ with $\chi (X)
= \omega_1$ and $s(X) = hL(X) = \beth_{\omega_1}$. This is still
another generalized $L$-space of Todor\v{c}evi\'c from the same
paper. By ($\star$), it does not have a point-$\omega_1$
$\pi$-base.

By a result of Magidor (see [3], Corollary 3), $V\models
2^{\beth_{\omega_1}} = (\beth_{\omega_1})^{++}$ is consistent,
relative to the existence of a supercompact cardinal.\footnote{The
author would like to thank Moti Gitik for the following comment:
``It is possible to replace the supercompact cardinal with the
strong cardinal, as was done by Segal and Merimovitch".} Force
with $Fn (\omega,\beth_{\omega_1})$ from $V$ as a ground model.
This will preserve $(\beth_{\omega_1})^+$ (in the form of
$(\aleph_1)^{V[G]}$) and all cardinals above it, while collapsing
all cardinals below it to $\aleph_0$. By a routine computation
(counting names and using the generic collapsing function),
$2^{\aleph_0} = \aleph_2$ in $V[G]$. We claim that the space $X$
from the ground model $V$ will possess in the generic extension
$V[G]$ all the properties of $Y$ stated above. As usual, the
topology of $Y$ is understood to be generated in $V[G]$ from
$\T_X$ as a base.

All topological base properties (i.e. those which can be
formulated in terms of a base and are invariant under choosing a
base) of $X$ will be inherited by $Y$. These include regularity,
left-separated structure, and ``$p$ is a complete accumulation
point of $A$," provided $A$ is in $V$ (even if the cardinal
$|A|^V$ is collapsed). The only property that needs an argument is
the hereditary Lindel\"{o}fness of $Y$. It is sufficient to check
that a set of size $\aleph_1$ in the extension contains a point of
complete accumulation of itself. Now, a set $A$ of cardinality
$\aleph_1$ in $V[G]$ has a name $\AA$ in $V$ indexed by the
ordinals in $(\beth_{\omega_1})^+$. Since our forcing poset has
size $\beth_{\omega_1}$, there is a single condition in $G$ which
evaluates in $V$ $(\beth_{\omega_1})^+$-many points of $\AA$, say
a set $B$. Now a complete accumulation point $b \in B$ of $B$ in
$X$ (which exists by $hL(X) < |B|$ in $V$), as we remarked, is the
same for $B \subseteq A$ in $V[G]$.

This space doesn't have a point-countable $\pi$-base for the
cardinal arithmetic reason alone (since it has $\chi(Y) = \F(Y) =
\aleph_0$ and $|Y| = \aleph_2 \leq \aleph_\omega$). This shows
that the cardinal assumption in the Corollary 2.15 (and {\it a
fortiori} in Theorem 2.13) is necessary. This example also gives a
negative answer to Problems 4.3 and 4.6 of [7].



\end{document}